\documentclass[11pt]{article}
\usepackage{geometry}                
\usepackage{graphicx}
\usepackage{enumerate}
\usepackage{setspace}
\usepackage{amssymb}
\usepackage{amsmath, amsthm}
\usepackage{amsmath}
\usepackage{amsthm}
\usepackage{blkarray}
\usepackage{epstopdf}
\usepackage{footnote}
\newtheorem{thm}{Theorem}

\newtheorem*{thma}{Theorem A}
\newtheorem*{thmb}{Theorem B}

\newtheorem{definition}{Definition}

\newtheorem{lemma}{Lemma}
\newtheorem{corollary}{Corollary}
\begin{document}\large{
\title{Inverse Theorem on Row Sequences of Linear Pad\'e-orthogonal Approximation}
\author{N. Bosuwan\thanks{The research of this author is supported by The Development and Promotion of Science and Technology Talents project (DPST) and Mahidol University.}\,\,\,\footnote{Corresponding author.}  and  G. L\'opez Lagomasino\thanks{The research of this author is supported by Ministerio de Economía y Competitividad under grant MTM2012-36372-C03-01.}   }
\maketitle

\begin{center}
{\emph{Dedicated to Professor Edward B. Saff on the occasion of his 70th birthday}}
\end{center}


\section*{Abstract}
We give necessary and sufficient conditions for the convergence with geometric rate of the denominators of linear Pad\'e-orthogonal approximants corresponding to a measure supported on a general compact set in the complex plane. Thereby, we obtain an analogue of Gonchar's theorem on row sequences of Pad\'e approximants.

\quad

 \section*{Keywords}{Pad\'e approximation, Pad\'e-orthogonal approximation, orthogonal polynomials,  Fourier-Pad\'e approximation, inverse problems.}

 \section*{2000 MSC.}
Primary 30E10, 41A27; Secondary 41A21.

\section{Introduction}

In this paper, $E$ denotes a compact subset of the complex plane $\mathbb{C}$ which contains infinitely many points such that $\overline{\mathbb{C}}\setminus E$ is simply connected. There exists a unique exterior conformal representation $\Phi$ from $\overline{\mathbb{C}}\setminus E$ onto $\overline{\mathbb{C}}\setminus \{w: |w|\leq 1\}$ satisfying $\Phi(\infty)=\infty$ and $\Phi'(\infty)>0.$ Thus
\[ \Phi(z) = \textup{cap}(E) z + \mathcal{O}(1), \qquad z \to \infty.
\]
It is well known that the constant $\textup{cap}(E)$ coincides with the logarithmic capacity of the compact set $E$ (see \cite[pag. 313]{Gol}).
Furthermore, we assume that $E$ is such that the inverse function $\Psi =\Phi^{-1}$ can be extended continuously to $\overline{\mathbb{C}}\setminus \{w: |w|< 1\}$  (the closure of a bounded Jordan region and a finite interval satisfy the above  conditions).

Let $\mu$ be a finite positive Borel measure with infinite support $\mbox{supp}(\mu)$ contained in $E$. We write $\mu \in \mathcal{M}(E)$ and define the associated inner product,
$$\langle g,h \rangle_{\mu}:=\int g(\zeta) \overline{h(\zeta)} d\mu(\zeta), \quad g,h \in L_2(\mu).$$
Let
$$p_{n}(z):=\kappa_n z^n+\cdots, \quad  \kappa_n>0,\quad n=0,1,\ldots,$$ be the orthonormal polynomial of degree $n$ with respect to $\mu$ having positive leading coefficient; that is, $\langle p_n, p_m \rangle_{\mu}=\delta_{n,m}.$  Denote by $\mathcal{H}(E)$ the space of all functions holomorphic in some neighborhood of $E.$

\begin{definition}\textup{
Let $F\in {\mathcal{H}}(E),\, \mu \in \mathcal{M}(E),$ and a pair of nonnegative integers $(n,m)$ be given. A rational function $[n/m]_F^{\mu}:=P_{n,m}^{\mu}/Q_{n,m}^\mu$ is called an $(n,m)$ \emph{(linear) Pad\'e-orthogonal approximant}  of $F$ with respect to $\mu$ if $P_{n,m}^\mu$ and $Q_{n,m}^\mu$ are polynomials satisfying
\begin{equation}\label{pade2}
\deg(P_{n,m}^\mu) \leq n, \quad \deg(Q_{n,m}^{\mu})\leq m,\quad Q_{n,m}^{\mu}\not\equiv 0,
\end{equation}
\begin{equation}\label{pade3}
\langle Q_{n,m}^{\mu} F-P_{n,m}^{\mu}, p_j \rangle_{\mu}=0, \quad \textup{for $j=0,1,\ldots,n+m.$}
\end{equation}
Since $Q_{n,m}^{\mu} \not\equiv 0$, we normalize it to have leading coefficient equal to $1$.}
\end{definition}

Obviously, given $Q_{n,m}^{\mu}$,
\[ P_{n,m}^{\mu}(z) = \sum_{j=0}^m\langle {Q_{n,m}^{\mu}F},p_j\rangle p_j(x)
\]
is uniquely determined.

It is easy to see that if $E=\{z\in \mathbb{C}: |z|\leq 1\}$ and $d\mu=d\theta/2\pi$ on the unit circle $\{z\in \mathbb{C}:|z|=1\}$, then the linear Pad\'e-orthogonal approximants are exactly the classical Pad\'e approximants. The concept of linear Pad\'e-orthogonal approximants was first introduced by H.J. Maehly \cite{maehly1960rational} in 1960. In fact, he considered linear Pad\'e-orthogonal approximants only for the case when $d\mu=dx/\sqrt{1-x^2}$ on $[-1,1]$. These rational functions are called Pad\'e-Chebyshev approximants (see \cite{baker1996pade}) or sometimes cross-multiplied approximants (see \cite{fleischer1972generalizations}). Later, E.W. Cheney defined linear Pad\'e-orthogonal approximants in a general setting ($E$ is not just a finite interval) in his book \cite{cheney1966introduction}. The study of linear Pad\'e-orthogonal approximants has mainly concentrated on the case when $\mu$ is supported on a finite interval (see e.g. \cite{Vib2006,Vib2009,gonchar1991convergence,gonchar1992rate,lubinsky1983convergence,Sutinpade,suetin1997asymptotics}). S.P. Suetin \cite{Sutinpade} was the first to prove the convergence of row sequences of linear Pad\'e-orthogonal approximants on $[-1,1]$ for a general class of measures for which the corresponding sequence of orthonormal polynomials has ratio asymptotic behavior. Moreover, he also proved an inverse result \cite{suetininverse} for row sequences of linear Pad\'e-orthogonal approximants on $[-1,1]$ under the assumption that the denominators of the approximants converge with geometric rate to a certain polynomial of degree $m$.  For measures satisfying Szeg\H{o}'s condition, V.I. Buslaev \cite{Vib2006,Vib2009} obtained inverse type results without the requirement that the denominators converge geometrically. Some problems on the convergence of diagonal sequences of linear Pad\'e-orthogonal approximants on $[-1,1]$ were considered in \cite{gonchar1991convergence,gonchar1992rate,lubinsky1983convergence,suetin1997asymptotics}. Some papers which consider measures $\mu$ supported on the unit circle are \cite{BLS,pade10,Vib2006,Vib2009}.   N. Bosuwan, G. L\'opez Lagomasino, and E.B. Saff gave in \cite{bosuwan} direct and inverse results for row sequences of linear Pad\'e-orthogonal approximants corresponding to measures supported on a general compact $E$ as described above (which we will discuss in details below). Note that linear Pad\'e-orthogonal approximants have also been called linear Pad\'e approximants of orthogonal expansions \cite{suetin2002Surveys}, Fourier-Pad\'e approximants \cite{BLS,pade10,CCL0}, and orthogonal Pad\'e approximants \cite{Vib2006,Vib2009}.

We would like to point out that there is another related construction called nonlinear Pad\'e approximants of orthogonal expansions (see \cite{nonlinear}). Unlike the classical case, these linear and nonlinear Pad\'e approximants of orthogonal expansions lead, in general, to different rational functions (see an example in \cite{nonlinear}). We will restrict our attention in this paper to linear Pad\'e-orthogonal approximants,  and in the sequel we will omit the word ``linear'' when we refer to them.

Let us introduce some notation. For any $\rho>1,$  we denote by
 $$\Gamma_{\rho}:=\{z\in \mathbb{C}: |\Phi(z)|=\rho\}, \quad \quad \mbox{and} \quad \quad D_{\rho}:=E\cup \{z\in \mathbb{C}: |\Phi(z)|<\rho\},$$
 a level curve of index $\rho$ and a canonical domain of index $\rho$, respectively.
We denote by $\rho_0(F)$ the index $\rho>1$ of the largest canonical domain $D_{\rho}$ to which $F$ can be extended as a holomorphic function, and by ${\rho_m(F)}$ the index $\rho$ of the largest canonical domain $D_{\rho}$ to which $F$ can be extended as a meromorphic function with at most $m$ poles (counting multiplicities).

Let $\mu \in \mathcal{M}(E)$ be such that
\begin{equation}\label{radius3}
\lim_{n \rightarrow \infty} |p_n(z)|^{1/n}=|\Phi(z)|,
\end{equation}
uniformly inside $\mathbb{C}\setminus E.$ Such measures are called \textit{regular} (cf.  \cite{totik}).  Here and in what follows, the phrase ``uniformly inside a domain'' means ``uniformly on each compact subset of the domain''.  The Fourier coefficient of $F$ with respect to $p_n$ is given by
\begin{equation}\label{Fourierco}
F_n:=\langle F,p_n\rangle_\mu =\int F(z) \overline{p_n(z)}d\mu(z).
\end{equation}
As for Taylor series (see, for example, \cite[Theorem 6.6.1]{totik}),  it is easy to show that
$$\rho_0(F)=\left(\varlimsup_{n \rightarrow \infty} |F_n|^{1/n} \right)^{-1}.$$
Additionally, the series $\sum_{n=0}^{\infty} F_n p_n(z)$ converges to $F(z)$ uniformly inside  ${D}_{\rho_{0}(F)}$ and diverges pointwise for all $z\in \mathbb{C}\setminus \overline{D_{\rho_0(F)}}.$
Therefore,
if \eqref{radius3} holds, then $$Q_{n,m}^{\mu}(z)F(z)-P_{n,m}^{\mu}(z)=\sum_{k=n+m+1}^{\infty} \langle Q_{n,m}^{\mu}F,p_k\rangle_{\mu}\,p_k(z)$$
for all $z\in D_{\rho_{0}(F)}$.

We showed in \cite[Example 1]{bosuwan} that $[n/m]_F^{\mu}$ is not unique in general. However, if $\mu$ satisfies the condition
\begin{eqnarray}\label{averbvkaer}
\Delta_{n,m}(F,\mu):=
\begin{vmatrix}
  \langle F, p_{n+1} \rangle_{\mu} &  \langle zF,p_{n+1}\rangle_{\mu} &  \cdots & \langle z^{m-1}F,p_{n+1}\rangle_{\mu} \\
   \vdots   & \vdots    & \vdots &  \vdots  \\
  \langle F, p_{n+m} \rangle_{\mu} & \langle z F,p_{n+m} \rangle_{\mu} &  \cdots &   \langle z^{m-1} F,p_{n+m}\rangle_{\mu} \\
 \end{vmatrix}\not=0
 \end{eqnarray}
or the condition that every solution of (\ref{pade2})-(\ref{pade3}) has $\deg Q_{n,m}^{\mu} = m$, then $[n/m]_F^{\mu}$ is unique.
Moreover, it is not difficult to verify that these two conditions are equivalent.

Let us introduce two classes of measures contained in $\mathcal{M}(E)$ which are relevant in what follows. We write $\mu \in \mathcal{R}(E)$ when the corresponding sequence of orthonormal polynomials has ratio asymptotics; that is,
\begin{equation}\label{ratioasym}
\lim_{n \rightarrow \infty} \frac{p_n(z)}{p_{n+1}(z)}=\frac{1}{\Phi(z)}.
\end{equation}
We say that Szeg\H{o} or strong asymptotics takes place, and write $\mu \in \mathcal{S}(E)$, if
\begin{equation}\label{szegoasym}
\lim_{n \rightarrow \infty} \frac{p_n(z)}{c_n \Phi^{n}(z)}=S(z) \quad \textup{and} \quad  \lim_{n\rightarrow \infty} \frac{c_n}{c_{n+1}}=1.
\end{equation}
The first limit in (\ref{szegoasym}) and the one in (\ref{ratioasym}) are assumed to hold uniformly inside  $\overline{\mathbb{C}}\setminus E$, the $c_n$'s are positive constants, and
$S(z)$ is a non-vanishing holomorphic function on $\overline{\mathbb{C}}\setminus E.$  Clearly, \eqref{szegoasym} $\Rightarrow$ \eqref{ratioasym} $\Rightarrow$  \eqref{radius3}.

These two classes of measures have been well studied when the measure $\mu$ is supported on an interval of the real line or the whole unit circle (see, for example \cite{Simon1} and \cite{Simon2}) and characterized in terms of the analytic properties of the measure or of the corresponding sequences of recurrence coefficients (in case of the real line) or the Verblunsky coefficients (for the unit circle).  For general compact sets $E$ contained in the complex plane the situation is not quite the same. There are many examples for which Szeg\H{o} asymptotics takes place for measures supported on a single Jordan curve or arc (see \cite{Kaliaguine1995,Erwin2009,PKSuetin,Szego1921,Szego1975,Widom1969})
 and polynomials orthogonal with respect to area type measures on a Jordan region (see \cite{Carleman,Erwin2008,Erwin2010,Stylianopoulos2013,PKSuetin1}). Outside of the previously mentioned cases of the segment and the unit circle, the only case fully described and easily verifiable where $\mathcal{R}(E)$ is substantially larger than $\mathcal{S}(E)$ is when $E$ is an arc of the unit circle, see \cite[Theorem 1]{Ba} and \cite[Theorem 1]{Be}. An interesting problem is to describe general measures in $\mathcal{R}(E)$ not necessarily in  $\mathcal{S}(E)$, for different compact sets $E$.

In \cite{bosuwan}, direct and inverse results for row sequences of Pad\'e-orthogonal approximants corresponding to a measure supported on a general compact set $E$  were proved. An analogue of Montessus de Ballore's theorem (direct result) for Pad\'e-orthogonal approximants is the following.

\begin{thma}\label{montessusana}
 Suppose $F\in \mathcal{H}(E)$ has poles of total multiplicity exactly $m$ in $D_{\rho_{m}(F)}$ at the (not necessarily distinct) points $\lambda_1,\ldots,\lambda_m$ and let $\mu \in \mathcal{R}(E)$. Then, $[n/m]_{F}^{\mu}$ is uniquely determined for all sufficiently large $n$ and the sequence $[n/m]_{F}^{\mu}$ converges uniformly to $F$ inside $D_{\rho_m(F)}\setminus \{\lambda_1,\ldots,\lambda_m\}$ as $n \rightarrow \infty.$ Moreover, for any compact subset $K$ of $D_{\rho_m(F)}\setminus\{\lambda_1,\ldots,\lambda_m\},$
\begin{equation}\label{supnroot}
\varlimsup_{n\rightarrow \infty} \|F-[n/m]_{F}^{\mu}\|^{1/n}_{K}\leq \frac{\max\{|\Phi(z)|:z\in K\}}{\rho_m(F)},
\end{equation}
where $\|\cdot \|_{K}$ denotes the sup-norm on $K$ and if $K\subset E,$ then $\max\{|\Phi(z)|:z\in K\}$ is replaced by $1.$ Additionally,
\begin{equation} \label{convzeros}
\varlimsup_{n\rightarrow \infty} \|Q_{n,m}^{\mu} - Q_m\|^{1/n} \leq \frac{\max\{|\Phi(\lambda_j)|:j=1,\ldots,m\}}{\rho_m(F)}<1,
\end{equation}
where $\|\cdot\|$ denotes the coefficient norm in the space of polynomials of degree at most $m$ and $Q_m(z) = \prod_{k=1}^{m} (z - \lambda_k)$.
\end{thma}

In the same paper \cite{bosuwan}, an inverse type result in the spirit of Suetin's theorem in  \cite{sue2} was also obtained. It states
\begin{thmb}\label{suetinana}
Let $F \in \mathcal{H}(E),$ $\mu \in \mathcal{S}(E)$, and  $m$ be a fixed nonnegative integer. If for all $n$ sufficiently large, $[n/m]_F^{\mu}$ has precisely $m$ finite poles $\lambda_{n,1}, \ldots, \lambda_{n,m},$ and
\begin{equation}
\lim_{n \rightarrow \infty} \lambda_{n,j}= \lambda_j, \quad j=1,2,\ldots,m,\notag
\end{equation}
($\lambda_1,\ldots, \lambda_m$ are not necessarily distinct), then
\begin{enumerate}
\item[$(i)$] $F$ is holomorphic in $D_{\rho_{\min}}$ where $\rho_{\min}:=\min_{1\leq j \leq m}|\Phi(\lambda_j) |$\textup{;}
\item[$(ii)$] $\rho_{m-1}(F)=\max_{1\leq j \leq m} |\Phi(\lambda_j)|$\textup{;}
\item[$(iii)$] $\lambda_1, \ldots, \lambda_m$ are singularities of $F;$ those lying in $D_{\rho_{m-1}(F)}$ are poles (counting multiplicities), and $F$ has no other poles in $D_{\rho_{m-1}(F)}$.
\end{enumerate}
\end{thmb}

In this paper, we prove a reciprocal of Theorem A for row sequences of Pad\'e-orthogonal approximants (see Theorem \ref{inversetheorem} below). As compared with Theorem B we must relax the condition on the measure to $\mu \in \mathcal{R}(E)$. To compensate, we will assume that the poles of the approximants converge with geometric rate as in \eqref{convzeros}. In contrast with Theorem B, we find that all the zeros of $Q_m$ are poles of $F$ and they all lie in $D_{\rho_m(F)}$.
Combining Theorem \ref{inversetheorem}  and Theorem A we obtain Corollary \ref{bana345} which characterizes the situation when $F$ has exactly $m$ poles in $D_{\rho_m(F)}$ (counting multiplicities) in terms of the exact rate of convergence in \eqref{convzeros}.  This corollary is an analogue of Gonchar's theorem  for row sequences of classical Pad\'e approximants (see e.g.  in \cite[Sect. 1]{CCL},  \cite[Sect. 3, Sect. 4]{gonchar1975convergence} or \cite[Sect. 2]{gonchar2011}).

An outline of the paper is as follows. In Sect.\,2, we state the main theorem and its corollary. All auxiliary lemmas are in Sect.\! 3. Sect. \!\!4 is devoted to the proof of the main result.


\section{Main results}


The main theorem is the following.
\begin{thm}\label{inversetheorem}
Let $F\in \mathcal{H}(E),$ $\mu\in \mathcal{R}(E),$ and $m$ be a fixed nonnegative integer. Assume that for all $n$ sufficiently large, $[n/m]_F^{\mu}$ has exactly $m$ finite poles and there exists a polynomial $Q_m(z)=\prod_{j=1}^m(z-\lambda_j)$ such that
\begin{equation}\label{banana51}
\varlimsup_{n\rightarrow \infty}\|Q_{n,m}^{\mu}-Q_m\|^{1/n} = \delta < 1.
\end{equation}
Then
\begin{equation}\label{banana75}
\rho_m(F) \geq \frac{1}{\delta} \max_{1\leq j \leq m} |\Phi(\lambda_j)|
\end{equation}
and in $D_{\rho_m(F)},$ the function $F$ has exactly $m$ poles at the points $\lambda_1,\ldots, \lambda_m.$
\end{thm}

In \cite[Theorem 1]{suetininverse}, S.P. Suetin proved this result for any measure $\mu$ supported on $[-1,1]$ such that $\mu' > 0$ almost everywhere on $[-1,1]$. Our proof of Theorem \ref{inversetheorem} is strongly influenced by the methods employed in that paper.

As a consequence of Theorem A and Theorem \ref{inversetheorem}, we immediately have the following corollary.

\begin{corollary}\label{bana345}
Let $F\in \mathcal{H}(E),$ $\mu\in \mathcal{R}(E),$ and $m$ be a fixed nonnegative integer. Then the following two assertions are equivalent:
\begin{enumerate}
\item[$(a)$] $F$ has exactly $m$ poles in $D_{\rho_m(F)}.$
\item[$(b)$] $[n/m]_F^{\mu}$ is uniquely determined and has precisely $m$ poles for all $n$ sufficiently large, and there exists a polynomial $Q_{m}$ of degree $m$ such that
$$\varlimsup_{n \rightarrow \infty} \|Q_{n,m}^{\mu}-Q_m\|^{1/n}=\delta<1.$$
\end{enumerate}
Moreover, if either (a) or (b) takes place, then the poles of $F$ in $D_{\rho_m(F)}$ coincide with the zeros $\lambda_1,\ldots, \lambda_m$ of $Q_m$ and
$$\delta=\frac{\max_{1\leq j \leq m} |\Phi(\lambda_j)|}{\rho_{m}(F)}.$$
\end{corollary}

\section{Auxiliary Lemmas}

The second type functions $s_n(z)$  defined by
\begin{equation*}
s_n(z):=\int \frac{\overline{p_n(\zeta)}}{z-\zeta} d\mu(\zeta), \quad z\in \overline{\mathbb{C}}\setminus \mbox{supp}(\mu),
\end{equation*}
play a major role in our proof. The first lemma connects the asymptotic behavior of the orthonormal polynomials $p_n$ and that of the second type functions $s_n.$


\begin{lemma}\label{secondtype}
If  $\mu \in \mathcal{R}(E),$ then $$\lim_{n \rightarrow \infty} p_n(z)s_n(z)=\frac{\Phi'(z)}{\Phi(z)},$$
uniformly inside $\overline{\mathbb{C}}\setminus E.$ Consequently, for any compact set $K\subset \mathbb{C}\setminus E,$ there exists $n_0$  such that  $s_n(z)\not=0$ for all $z\in K$ and $n \geq n_0.$
\end{lemma}
\begin{proof}[Proof of Lemma \ref{secondtype}] See Lemma 3.1 in \cite{bosuwan}.
\end{proof}

Recall that $\kappa_n$ is the leading coefficient of the orthonormal polynomial $p_n.$ The second lemma shows that under the condition $\mu\in \mathcal{R}(E),$ the limit of the ratios of $\kappa_n$ is the capacity of $E$.

\begin{lemma}\label{kappa}
If  $\mu \in \mathcal{R}(E),$ then
$$\lim_{n \rightarrow \infty}\frac{\kappa_n}{\kappa_{n+1}}=\textup{cap}(E),$$
where $\textup{cap}(E)$ is the capacity of $E.$
\end{lemma}
\begin{proof}[Proof of Lemma \ref{kappa}]
Since $\lim_{n \rightarrow \infty} {zp_n(z)}/{p_{n+1}(z)}={z}/{\Phi(z)}$ holds uniformly inside  $\overline{\mathbb{C}}\setminus E,$ then
$$\lim_{n\rightarrow \infty} \frac{\kappa_n}{\kappa_{n+1}}=\lim_{n \rightarrow \infty} \lim_{z \rightarrow \infty} \frac{z p_n(z)}{p_{n+1}(z)}=  \lim_{z \rightarrow \infty} \lim_{n \rightarrow \infty} \frac{z p_n(z)}{p_{n+1}(z)}=\lim_{z \rightarrow \infty}\frac{z}{\Phi(z)}=\textup{cap}(E).$$
\end{proof}

The next lemma is a curious relation of complex numbers which we will use at the end of the proof of Theorem 1.

\begin{lemma}\label{trick}
 Let $N_0\in \mathbb{N}$ and $C>0.$ If a sequence of complex numbers $\{F_N\}_{N \in \mathbb{N}}$ has the following properties:
\begin{enumerate}
\item[$(i)$] $\lim_{N \rightarrow \infty} |F_N|^{1/N}=0,$
\item[$(ii)$] $|F_N|\leq C \sum_{k=N+1}^{\infty} |F_k|,$ for all $N \geq N_0$,
\end{enumerate}
then there exists $N_1\in \mathbb{N}$ such that $F_N=0$ for all $N \geq N_1.$

\end{lemma}

\begin{proof}[Proof of Lemma \ref{trick}] Given the assumptions, there exists $M$ such that for all
$ N\ge M,$
$$ |F_N|^{1/N}<\frac{1}{C+2}, \quad \quad \textup{and} \quad \quad  |F_N|\le C\sum_{k=N+1}^\infty|F_k|.$$ We claim that for those $N$'s,
 $$|F_N|\le\left(\frac{C}{C+1}\right)^n\left(\frac{1}{C+2}\right)^N$$
 for any non-negative integer $n$. Then, letting $n\rightarrow \infty,$ we see that $|F_N|=0.$

To prove the claim, we use induction on $n.$ When $n=0,$ the formula follows immediately from $|F_N|^{1/N}<1/(C+2).$ In general, using induction  it follows that
 $$|F_N|\le C\sum_{k=N+1}^\infty|F_k|\le C\sum_{k=N+1}^{\infty}\left(\frac{C}{C+1}\right)^n\left(\frac{1}{C+2}\right)^k=\left(\frac{C}{C+1}\right)^{n+1}\left(\frac{1}{C+2}\right)^N.$$
 This completes the proof.
\end{proof}


\section{Proof of Theorem \ref{inversetheorem}}


In the proof of Theorem 1, we mainly use the asymptotic properties of the orthogonal polynomials $p_n$ and the second type functions $s_n$ listed below.

From \eqref{ratioasym}, it follows that
\begin{eqnarray}\label{banana7}
\lim_{n \rightarrow \infty} \frac{p_{n}(z)}{p_{n+l}(z)}=\frac{1}{\Phi(z)^l}, \quad l=0,1,\ldots,
\end{eqnarray}
uniformly inside $\overline{\mathbb{C}} \setminus E$. By (\ref{banana7}) and Lemma \ref{secondtype} for any $l,p=0,1,\ldots,$ we have
\begin{align}\label{banana8}
\lim_{n \rightarrow \infty} \frac{s_{n+l}(z)}{s_{n}(z)}&=\lim_{n \rightarrow \infty} \frac{p_n(z)}{p_{n+l}(z)}\frac{p_{n+l}(z)s_{n+l}(z)}{p_n(z)s_n(z)}=\frac{1}{\Phi(z)^{l}} \frac{\Phi'(z)/\Phi(z)}{\Phi'(z)/\Phi(z)}=\frac{1}{\Phi(z)^{l}} ,
\end{align}
uniformly inside $\overline{\mathbb{C}} \setminus E.$ Furthermore,
\begin{eqnarray}\label{banana9}
\lim_{n \rightarrow \infty} |p_n(z)|^{1/n}=|\Phi(z)|
\end{eqnarray}
and
\begin{eqnarray}\label{banana10}
\lim_{n \rightarrow \infty} |s_n(z)|^{1/n}=|\Phi(z)|^{-1}
\end{eqnarray}
uniformly inside $\mathbb{C} \setminus E,$ are  trivial consequences of (\ref{banana7}) and (\ref{banana8}).


\begin{proof}[Proof of Theorem \ref{inversetheorem}]

We organize the proof of Theorem \ref{inversetheorem} as follows. First of all, we assume that in the region $D_{\rho_m(F)},$ the function $F$ has $k < m$ poles  in $D_{\rho_m(F)}$ at the points $\tilde{\lambda}_1,\ldots, \tilde{\lambda}_k$. Set
$$Q_{m}(z)=\prod_{j=1}^m(z-\lambda_j),  \quad \quad\quad  \tilde{Q}_k(z):=\prod_{j=1}^k (z-\tilde{\lambda}_j)=\prod_{w=1}^{\gamma}(z-\tilde{\alpha}_w)^{k_w},$$
where $\tilde{\alpha}_1,\ldots,\tilde{\alpha}_k$ are distinct and $\sum_{w=1}^{\gamma}k_w=k$.  Arguing as in the proof of  \cite[Theorem 2.1]{bosuwan}, our assumptions imply that  the
sequence $[n/m]_F^{\mu}(z)$ converges in capacity to $F$ inside $D_{\rho_m(F)},$ as $n\rightarrow \infty.$ More precisely, for any $\varepsilon > 0$ and any compact subset $K \subset D_{\rho_m(F)}$
\[ \lim_{n \to \infty} \textup{cap}(\{z \in K: |F(z) - [n/m]_F^{\mu}(z)| \geq \varepsilon\}) = 0.
\]
By Gonchar's lemma (see \cite[Lemma 1]{goncharcap} on page 507 and sentence at the beginning of that page regarding Cartan's inequality as well as the translator's correction), this implies that each $\tilde{\alpha}_w$ attracts at least $k_w$ poles of $[n/m]_F^{\mu}$ as $n \rightarrow \infty.$ From this and  \eqref{banana51} it follows that $\tilde{\lambda}_1,\ldots, \tilde{\lambda}_k$ are zeros of $Q_m$ and $Q_{m}F$ is holomorphic in $D_{\rho_m(F)}.$ We can reindex $\lambda_j,\tilde{\lambda}_j,\tilde{\alpha}_w$ so that
$$\lambda_j=\tilde{\lambda}_j,\quad \quad j=1,\ldots,k,\quad \quad  \textup{and}  \quad \quad \sigma_w:=|\Phi(\tilde{\alpha}_w)|, \quad w=1,\ldots, \gamma,$$
$$|\Phi(\lambda_1)|\leq \ldots \leq |\Phi(\lambda_k)|,\quad \quad \textup{and} \quad \quad  1\leq \sigma_1\leq \ldots \leq \sigma_\gamma.$$
Next, we will prove by contradiction that  the assumption $k <m$ on the number of poles of $F$ in $D_{\rho_m(F)},$ implies that $D_{\rho_{m}(F)}=\mathbb{C},$ i.e. $\rho_m(F)=\infty.$ To this end we show that if $\rho_m(F) <\infty$ then $F$ has at most $m$ poles in a canonical region which is strictly larger than $D_{\rho_m(F)}$ which  clearly contradicts the definition of $\rho_{m}(F).$ This step is the main part of the proof of Theorem 1. Finally, we show that  if $\rho_m(F)=\infty$ and $k < m$ then $F$ is a rational function with less than $m$ poles which  contradicts the assumption that for all $n$ sufficiently large all the polynomials $Q_{n,m}^{\mu}$ have degree $m.$ Thus, $F$ must have exactly $m$ poles in $D_{\rho_m(F)}$ and using Theorem A it follows that they must be the  points $\lambda_1,\ldots, \lambda_m$ counting multiplicities.

Let us suppose that $F$ has $k < m$ poles at the points  $\tilde{\lambda}_1,\ldots, \tilde{\lambda}_k$ in $D_{\rho_m(F)}$. The indices are taken so that  $\tilde{\lambda}_j=\lambda_j, j=1,\ldots,k$. Let us prove that $D_{\rho_{m}(F)}=\mathbb{C}$. To the contrary, assume that $\rho_m(F)< \infty.$ We plan to show that
\begin{equation}\label{banana60}
\varlimsup_{n \rightarrow \infty} \left|\left[F Q_m\right]_n \right|^{1/n} \leq \frac{\delta}{\rho_m(F)}.
\end{equation}
(Recall that $[FQ_{m}]_n:=\langle FQ_m, p_n \rangle_{\mu}.$)
Combining this and \eqref{banana9}, it follows that $Q_{m}F$ is holomorphic in $D_{\sigma},$ where $\sigma=\rho_m(F)/\delta < \rho_m(F).$ This implies that $F$ is meromorphic with at most $m$ poles in $D_{\sigma}$ which contradicts the definition of $\rho_m(F)$.

Now, let us prove \eqref{banana60}. By the definition of Pad\'e-orthogonal approximants, we have
$$F(z)Q_m(z) =F(z)(Q_m(z)-Q_{n,m}^{\mu}(z))+ P_{n,m}^{\mu}(z)+\sum_{\nu=n+m+1}^{\infty} \langle Q_{n,m}^{\mu}F,\, p_{\nu} \rangle_{\mu} p_\nu(z),$$
which implies
\begin{equation}\label{banana11}
[F Q_{m}]_{n+b}=[F (Q_m-Q_{n,m}^{\mu})]_{n+b}, \quad \quad b=1,\ldots,m.
\end{equation}
Applying Cauchy's residue theorem to the function $F (Q_m-Q_{n,m}^{\mu})$ on the closed region $\overline{D}_{\rho}\setminus D_r:=\{z\in \mathbb{C}: r \leq |\Phi(z)|\leq \rho\},$ $1<r < \sigma_1, \sigma_\gamma< \rho< \rho_m,$  we obtain
\begin{align}
&[F (Q_m-Q_{n,m}^{\mu})]_{n+b}=\frac{1}{2\pi i}\int_{\Gamma_r} F(t) (Q_m(t)-Q_{n,m}^{\mu}(t))s_{n+b}(t) dt \notag \\
&=\frac{1}{2 \pi i} \int_{\Gamma_\rho} F(t) (Q_m(t)-Q_{n,m}^{\mu}(t))s_{n+b}(t) dt -\sum_{w=1}^{\gamma} \textup{res}(F (Q_m-Q_{n,m}^{\mu})s_{n+b},\tilde{\alpha}_w)\notag\\
&=\frac{1}{2 \pi i} \int_{\Gamma_\rho} F(t) (Q_m(t)-Q_{n,m}^{\mu}(t))s_{n+b}(t) dt+\sum_{w=1}^{\gamma} \textup{res}(F Q_{n,m}^{\mu}s_{n+b},\tilde{\alpha}_w).\label{banana17}
\end{align}
Note that $F Q_{n,m}^{\mu}s_{n+b}$ is meromorphic on $\overline{D}_{\rho}\setminus D_{r}$ and has a pole at $\tilde{\alpha}_w$ of multiplicity at most $k_w$ for each $w=1,\ldots, \gamma$. Using the limit formula for residue, we have
\begin{equation}
\textup{res}(FQ_{n,m}^{\mu}s_{n+b},\tilde{\alpha}_w)=\frac{1}{(k_w-1)!}\lim_{z \rightarrow \tilde{\alpha}_w} ((z-\tilde{\alpha}_w)^{k_w} F(z) Q_{n,m}^{\mu}(z) s_{n+b}(z) )^{(k_w-1)}.
\end{equation}
By the Leibniz formula and the fact that for $n$ sufficiently large, $s_n(z)\not=0$ for $z\in \mathbb{C}\setminus E$ (see Lemma \ref{secondtype}), we can transform the expression under the limit sign as follows
\begin{align}\label{banana40}
&((z-\tilde{\alpha}_w)^{k_w} F(z) Q_{n,m}^{\mu}(z) s_{n+b}(z) )^{(k_w-1)}=\left((z-\tilde{\alpha}_w)^{k_w} F(z) Q_{n,m}^{\mu}(z) s_n(z) \frac{s_{n+b}(z)}{s_n(z)} \right)^{(k_w-1)}\notag\\
&=\sum_{p=0}^{k_w-1} {k_w-1 \choose p} ((z-\tilde{\alpha}_w)^{k_w}Q_{n,m}^{\mu}(z) F(z) s_n(z))^{(k_w-1-p)} \left(\frac{s_{n+b}}{s_{n}} \right)^{(p)}(z).
\end{align}
To avoid long expressions, let us introduce the following notation
$$\beta_n(w,p):=\frac{1}{(k_w-1)!}{k_w-1 \choose p} \lim_{z \rightarrow \tilde{\alpha}_w} ((z-\tilde{\alpha}_w)^{k_w}Q_{n,m}^{\mu}(z) F(z) s_n(z))^{(k_w-1-p)},$$
for $w=1,\ldots, \gamma$ and $p=0,\ldots, k_w-1$ and
\begin{equation}
\label{e}\eta(\rho,n,b):=\frac{1}{2 \pi i} \int_{\Gamma_\rho} F(t) (Q_m(t)-Q_{n,m}^{\mu}(t))s_{n+b}(t) dt,\quad \quad b=1,\ldots, m
\end{equation}
(notice that the $\beta_n(w,p)$ do not depend on $b$). So, by \eqref{banana11} we can rewrite  (\ref{banana17}) as
\begin{equation}\label{banana13}
[F Q_m]_{n+b}=\eta(\rho,n,b)+\sum_{w=1}^{\gamma} \sum_{p=0}^{k_w-1}  \beta_n(w,p) \left(\frac{s_{n+b}}{s_{n}} \right)^{(p)}(\tilde{\alpha}_w) , \quad b=1,\ldots,m.
\end{equation}
Since $k\leq m-1,$ we have
\begin{equation}\label{banana12}
\sum_{w=1}^{\gamma} \sum_{p=0}^{k_w-1}  \beta_n(w,p) \left(\frac{s_{n+b}}{s_{n}} \right)^{(p)}(\tilde{\alpha}_w)
=[F Q_m]_{n+b}-\eta(\rho,n,b), \quad b=1,\ldots,k.
\end{equation}
We will view \eqref{banana12} as a system of $k$ equations on the $k$ unknowns $\beta_n(w,p).$ If we can show that

\begin{equation}
\begin{array}{c}
{ \begin{array}{c}
    {w=1,\ldots,\gamma}
\end{array}}\\[5pt]
\Lambda_n=\begin{vmatrix}\label{matrixL}
  \left(\frac{s_{n+1}}{s_n}\right)(\tilde{\alpha}_w) &  \left(\frac{s_{n+1}}{s_n}\right)'(\tilde{\alpha}_w) &  \cdots & \left(\frac{s_{n+1}}{s_n}\right)^{(k_w-1)}(\tilde{\alpha}_w)  \\
  \left(\frac{s_{n+2}}{s_n}\right)(\tilde{\alpha}_w) &  \left(\frac{s_{n+2}}{s_n}\right)'(\tilde{\alpha}_w) &  \cdots &\left(\frac{s_{n+2}}{s_n}\right)^{(k_w-1)}(\tilde{\alpha}_w)  \\
   \vdots   & \vdots    & \vdots &  \vdots  \\
  \left(\frac{s_{n+k}}{s_n}\right)(\tilde{\alpha}_w) & \left(\frac{s_{n+k}}{s_n}\right)'(\tilde{\alpha}_w) &  \cdots &  \left(\frac{s_{n+k}}{s_n}\right)^{(k_w-1)}(\tilde{\alpha}_w)  \\
 \end{vmatrix}\not=0
  \end{array}
\end{equation}
(this expression represents the determinant of order $k$ in
which the indicated group of columns are successively written out for $w= 1,\ldots,\gamma$),
then we can express $\beta_n(w,p)$ in terms of $(s_{n+b}/s_n)^{(p)}(\tilde{\alpha}_w)$ and $[F Q_m]_{n+b}-\eta(\rho,n,b).$ However, from \eqref{banana8} and the Weierstrass theorem it follows that
\begin{eqnarray*}
&\lim_{n \rightarrow \infty} \Lambda_n=\Lambda:=
\begin{array}{c}
{ \begin{array}{c}
    {w=1,\ldots,\gamma}
\end{array}}\\[5pt]
\begin{vmatrix}
  R(\tilde{\alpha}_w) &  R'(\tilde{\alpha}_w) &  \cdots & R^{(k_w-1)}(\tilde{\alpha}_w)  \\
 R^2(\tilde{\alpha}_w) & (R^2)'(\tilde{\alpha}_w) &  \cdots &(R^2)^{(k_w-1)}(\tilde{\alpha}_w)  \\
   \vdots   & \vdots    & \vdots &  \vdots  \\
 R^k(\tilde{\alpha}_w) & (R^k)'(\tilde{\alpha}_w) &  \cdots &  (R^k)^{(k_w-1)}(\tilde{\alpha}_w)  \\
 \end{vmatrix}
  \end{array}
 \end{eqnarray*}
\begin{align}\label{banana18}
 &=\prod_{w=1}^{\gamma}(k_w-1)!! \prod_{w=1}^{\gamma} (-\Phi'(\tilde{\alpha}_w))^{k_w(k_w-1)/2} \prod_{w=1}^{\gamma} \Phi(\tilde{\alpha}_w)^{-k_w^2} \prod_{1\leq i< j\leq \gamma} \left(\frac{1}{\Phi(\tilde{\alpha}_j)}-\frac{1}{\Phi(\tilde{\alpha}_i)} \right)^{k_i k_j}\notag \\
 &\not=0,
\end{align}
\noindent where $R(z)=1/\Phi(z)$ and $n!!=0!1!\cdots$n!
(using for example \cite[Theorem 1]{Sobczyk} for proving the last equality), for sufficiently large $n,$ $\Lambda_n\not=0.$ Therefore, for all sufficiently large $n$, $|\Lambda_n|\geq c_1>0$ and we will only consider such $n$ below. Hereafter,  $c_1,c_2,c_3,\ldots$ denote absolute constants which do not depend on $n$.

Applying Cramer's rule to \eqref{banana12}, we have
\begin{equation}\label{banana14}
\beta_n(w,p)=\frac{\Lambda_n(w,p)}{\Lambda_n}=\frac{1}{\Lambda_n}\sum_{j=1}^k (-1)^{j+q(w,p)}\left( [F Q_m]_{n+j}-\eta(\rho,n,j)\right) M_n(j,q(w,p)),
\end{equation}
where  $\Lambda_n(w,p)$ is the determinant obtained from $\Lambda_n$ replacing the column with index $q(w,p):=(\sum_{l=0}^{w-1} k_{l})+p+1$ (where we define $k_0:=0$) with the column

$$[([F Q_m]_{n+1}-\eta(\rho,n,1))\,\,\,\quad \ldots \quad  \,\,\,([F Q_m]_{n+k}-\eta(\rho,n,k)) ]^{T}$$
and $M_n(j,q)$ is the $(j,q)^{\textup{th}}$ minor of $\Lambda_n(w,p).$ Substituting $\beta_n(w,p)$ in the formula (\ref{banana13}) with the expression in (\ref{banana14}) for $b=k+1$ , we obtain
\begin{align}\label{banana15}
&[F Q_m]_{n+k+1}=\eta(\rho,n,k+1)+ \\
&\frac{1}{\Lambda_n}\sum_{w=1}^{\gamma} \sum_{p=0}^{k_w-1}  \sum_{j=1}^k (-1)^{j+q(w,p)}\left( [F Q_m]_{n+j}-\eta(\rho,n,j)\right) M_n(j,q(w,p)) \left(\frac{s_{n+k+1}}{s_{n}} \right)^{(p)}(\tilde{\alpha}_w) \notag.
\end{align}

Let us transform the triple sum on the right side of the last expression
\begin{align}\label{banana150}
&\frac{1}{\Lambda_n}\sum_{w=1}^{\gamma} \sum_{p=0}^{k_w-1}  \sum_{j=1}^k (-1)^{j+q(w,p)}\left( [F Q_m]_{n+j}-\eta(\rho,n,j)\right) M_n(j,q(w,p)) \left(\frac{s_{n+k+1}}{s_{n}} \right)^{(p)}(\tilde{\alpha}_w)\notag\\
&=\frac{1}{\Lambda_n}     \sum_{j=1}^k\left( [F Q_m]_{n+j}-\eta(\rho,n,j)\right)      \sum_{w=1}^{\gamma} \sum_{p=0}^{k_w-1} (-1)^{j+q(w,p)}M_n(j,q(w,p))\left(\frac{s_{n+k+1}}{s_{n}} \right)^{(p)}(\tilde{\alpha}_w)\notag\\
&=\frac{1}{\Lambda_n} \sum_{j=1}^k\left( [F Q_m]_{n+j}-\eta(\rho,n,j)\right) \tilde{\Lambda}_n(j,k+1)\notag\\
&=\frac{1}{\Lambda_n} \sum_{j=1}^k [F Q_m]_{n+j}\tilde{\Lambda}_n(j,k+1)-\frac{1}{\Lambda_n} \sum_{j=1}^k\eta(\rho,n,j)  \tilde{\Lambda}_n(j,k+1),
\end{align} where we denote by $\tilde{\Lambda}_n(j,k+1)$ the determinant obtained from $\Lambda_n$ replacing the $j$th row by the row
$$
\begin{array}{c}
{ \begin{array}{c}
    {w=1,\ldots,\gamma}
\end{array}}\\[5pt]
\left[\left(\frac{s_{n+k+1}}{s_{n}}\right)(\tilde{\alpha}_w)\quad  \left(\frac{s_{n+k+1}}{s_{n}}\right)^{(1)}(\tilde{\alpha}_w)\quad  \quad\ldots \quad\quad \left(\frac{s_{n+k+1}}{s_{n}}\right)^{(k_w-1)}(\tilde{\alpha}_w)\right].
  \end{array}
$$
Define
$$a_{j,n}:=-\frac{\tilde{\Lambda}_n(j,k+1)}{\Lambda_n}, \quad \quad \textup{for $j=1,\ldots,k$}, \quad \quad \textup{and}\quad \quad  a_{k+1,n}:=1.$$
Therefore, we have
\begin{equation}\label{banana21}
\sum_{j=1}^{k+1} a_{j,n} [F Q_m]_{n+j}=\sum_{j=1}^{k+1} a_{j,n} \eta(\rho,n,j).
\end{equation}

Let us obtain some lower and upper bounds for $|a_{j,n}|,$ for $j=1,\ldots,k.$ Since
     $$\lim_{n \rightarrow \infty} |a_{1,n}|=\left|\frac{\begin{array}{c}
{ \begin{array}{c}
    {w=1,\ldots,\gamma}
\end{array}}\\[5pt]
\begin{vmatrix}
 R^2(\tilde{\alpha}_w) & (R^2)'(\tilde{\alpha}_w) &  \cdots &(R^2)^{(k_w-1)}(\tilde{\alpha}_w)  \\
   \vdots   & \vdots    & \vdots &  \vdots  \\
 R^k(\tilde{\alpha}_w) & (R^k)'(\tilde{\alpha}_w) &  \cdots &  (R^k)^{(k_w-1)}(\tilde{\alpha}_w)  \\
 R^{k+1}(\tilde{\alpha}_w) &  (R^{k+1})'(\tilde{\alpha}_w) &  \cdots & (R^{k+1})^{(k_w-1)}(\tilde{\alpha}_w)  \\
 \end{vmatrix}
  \end{array}}{\begin{array}{c}
{ \begin{array}{c}
    {w=1,\ldots,\gamma}
\end{array}}\\[5pt]
\begin{vmatrix}
 R(\tilde{\alpha}_w) & R'(\tilde{\alpha}_w) &  \cdots &R^{(k_w-1)}(\tilde{\alpha}_w)  \\
   \vdots   & \vdots    & \vdots &  \vdots  \\
 R^{k-1}(\tilde{\alpha}_w) & (R^{k-1})'(\tilde{\alpha}_w) &  \cdots &  (R^{k-1})^{(k_w-1)}(\tilde{\alpha}_w)  \\
 R^{k}(\tilde{\alpha}_w) &  (R^{k})'(\tilde{\alpha}_w) &  \cdots & (R^{k})^{(k_w-1)}(\tilde{\alpha}_w)  \\
 \end{vmatrix}
  \end{array}}\right|$$
  $$=\frac{\left|\prod_{w=1}^{\gamma} \left((k_w-1)!!  (\Phi'(\tilde{\alpha}_w))^{k_w(k_w-1)/2}  \Phi(\tilde{\alpha}_w)^{-k_w(k_w+1)}\right) \prod_{1\leq i< j\leq \gamma} \left(1/\Phi(\tilde{\alpha}_j)-1/\Phi(\tilde{\alpha}_i) \right)^{k_i k_j}\right|}{\left|\prod_{w=1}^{\gamma}  \left((k_w-1)!!  (\Phi'(\tilde{\alpha}_w))^{k_w(k_w-1)/2}  \Phi(\tilde{\alpha}_w)^{-k_w^2}\right) \prod_{1\leq i< j\leq \gamma} \left(1/\Phi(\tilde{\alpha}_j)-1/\Phi(\tilde{\alpha}_i) \right)^{k_i k_j}\right|}$$
  $$=\prod_{w=1}^{\gamma} |(\Phi(\tilde{\alpha}_w))^{-k_w}|\not=0,$$
it follows that  for $n$ sufficiently large,
 \begin{equation}\label{banana19}
 0<c_2\leq|a_{1,n}|\leq c_3.
 \end{equation}
Analogously, one obtains that
\begin{equation}\label{banana20}
|a_{j,n}|\leq c_4, \quad \quad j=2,3,\ldots,k.
\end{equation}
The inequalities \eqref{banana19} and \eqref{banana20} will be used later.

In \eqref{banana21}, we replace the index $n$ by the indices $n+\nu,\, \nu=0,1,\ldots, \nu_0-1,$ where $\nu_0$ is an arbitrary natural number  greater than $3k+1$. Then, we have
\begin{equation}\label{banana22}
 \sum_{j=1}^{k+1} a_{j,n+\nu} [F Q_m]_{n+\nu+j}=\sum_{j=1}^{k+1} a_{j,n+\nu} \eta(\rho,n+\nu,j), \quad \quad \nu=0,1,\ldots, \nu_0-1.
 \end{equation}
We rewrite the system of equations \eqref{banana22} in the following form
\begin{equation}\label{banana23}
\sum_{j=1}^{\min(\nu_0-\nu,k+1)} a_{j,n+\nu} [F Q_m]_{n+\nu+j}=B_{n+\nu}(\rho),\quad \quad \nu=0,1,\ldots, \nu_0-1,
\end{equation}
where
$$B_{n+\nu}(\rho):=\sum_{j=1}^{k+1} a_{j,n+\nu} \eta(\rho,n+\nu,j), \quad \quad \nu=0,1,\ldots,\nu_0-k-1,$$
$$B_{n+\nu}(\rho):=\sum_{j=1}^{k+1} a_{j,n+\nu} \eta(\rho,n+\nu,j)-\sum_{j=\nu_0-\nu+1}^{k+1} a_{j,n+\nu} [F Q_m]_{n+\nu+j}, \quad \nu=\nu_0-k,\ldots, \nu_0-1.$$
We view this as a system of $\nu_0$ equations on the $\nu_0$ unknowns $[F Q_m]_{n+1},\ldots, [F Q_m]_{n+\nu_0}.$ Notice that the matrix corresponding to the system is upper triangular and its determinant equals
$$\Lambda_n^*(\nu_0)=\prod_{\nu=0}^{\nu_0-1} a_{1, n+\nu}\not=0,$$
for all $n$ sufficiently large (see \eqref{banana19}). Therefore,
$[F Q_m]_{n+1}=\Lambda^*_n(\nu_0,1)/\Lambda^*_n(\nu_0),$
where $\Lambda^*_n(\nu_0,1)$ is the determinant obtained replacing the first column of the determinant of the system by the column
$[B_{n}(\rho)\,\,\, \cdots \,\,\, B_{n+\nu_0-1}(\rho)]^{T}.$ Expanding $\Lambda^*_n(\nu_0,1)$ by the first column, we get
$$\Lambda^*_n(\nu_0,1)=\sum_{\nu=0}^{\nu_0-1} (-1)^{\nu} B_{n+\nu} (\rho) D(n,\nu_0,\nu),$$
where $D(n,\nu_0,\nu)$ is the $(\nu+1,1)^{\textup{th}}$ minor of $\Lambda^*_n(\nu_0,1).$ Moreover, it is easy to check that
\begin{equation}\label{banana24}
D(n,\nu_0,\nu)=D(n,\nu+1,\nu) \prod_{j=\nu+1}^{\nu_0-1} a_{1,n+j},
\end{equation}
and we denote $D(n,\nu):=D(n,\nu+1,\nu).$ Therefore, by \eqref{banana24}, we have
\begin{equation}\label{banana30}
[Q_{m}F]_{n+1}=\sum_{\nu=0}^{\nu_0-1}(-1)^{\nu} B_{n+\nu}(\rho)\frac{D(n,\nu_0,\nu)}{\prod_{j=0}^{\nu_0-1} a_{1,n+j}}=\sum_{\nu=0}^{\nu_0-1}(-1)^{\nu} B_{n+\nu}(\rho)\frac{D(n,\nu)}{\prod_{j=0}^{\nu} a_{1,n+j}}.
\end{equation}

From \eqref{banana19} and \eqref{banana20}, we get
$$\left|\frac{D(n,\nu)}{\prod_{j=0}^{\nu} a_{1,n+j}}\right| \leq c_5, \quad \quad \nu=0,1,\ldots,3k.$$
Our next goal is to estimate
 $$\left|\frac{D(n,\nu)}{\prod_{j=0}^{\nu} a_{1,n+j}}\right|, \quad \quad \nu=3k+1,3k+2,\ldots,\nu_0-1.$$
For this purpose, we expand the determinant
 $D(n,\nu)$ along the first row. We have
 $$D(n,\nu)=\sum_{p=1}^k (-1)^{p+1} a_{p+1,n} M_{n}(\nu,1,p),$$
 where $M_{n}(\nu,1,p)$ is the $(1,p)^{\textup{th}}$ minor of $D(n,\nu).$
 It is easy to check that for $\nu\geq k+1,$
 $$M(\nu,1,p)=\frac{1}{a_{1,n}} \prod_{j=0}^{p-1}a_{1,n+j} D(n+p,\nu-p).$$
 Hence,
$$D(n,\nu)=\frac{1}{a_{1,n}} \sum_{p=1}^k (-1)^{p+1} a_{p+1,n} \prod_{j=0}^{p-1}a_{1,n+j} D(n+p,\nu-p).$$
Replacing $n$ by $n+r$ and $\nu$ by $\nu-r$, where $r=0,1,2,\ldots, \nu-2k,$ we obtain the following relations
$$a_{1,n+r}D(n+r,\nu-r)=\sum_{p=1}^k (-1)^{p+1} a_{p+1,n+r} \prod_{j=r}^{r+p-1}a_{1,n+j} D(n+r+p,\nu-r-p).$$
Dividing both sides  by $\prod_{j=r}^{\nu-1}a_{1,n+j},$
we get
\begin{equation} \label{eq:*}
a_{1,n+r}\frac{D(n+r,\nu-r)}{\prod_{j=r}^{\nu-1} a_{1,n+j}}=\sum_{p=1}^k (-1)^{p+1} a_{p+1,n+r}\frac{D(n+r+p,\nu-r-p)}{\prod_{j=r+p}^{\nu-1}a_{1,n+j}}.
\end{equation}
For fixed $n$ and $\nu,$ the quantity
$$V_{\nu-(r+p)}:=(-1)^{r+p} \frac{D(n+p+r,\nu-r-p)}{\prod_{j=r+p}^{\nu-1}a_{1,n+j}}$$ depends only on the sum $r+p.$ With this notation \eqref{eq:*} can be rewritten as
$$\sum_{p=0}^{k} a_{p+1,n+r} V_{\nu-(r+p)}=0,\quad \quad r=0,1,\ldots, \nu-2k.$$
 Setting $\tilde{\Lambda}_{n+r}(k+1,k+1)=-\Lambda_{n+r},$ we bring these relations to the form
 $$\sum_{p=0}^{k} \left( -\frac{\tilde{\Lambda}_{n+r}(p+1,k+1)}{\Lambda_{n+r}}\right)V_{\nu-(r+p)}=0$$
or, what is the same,
\begin{equation}\label{banana61}
 \sum_{p=0}^{k} \tilde{\Lambda}_{n+r}(p+1,k+1)V_{\nu-(r+p)}=0, \quad \quad r=0,1,\ldots, \nu-2k.
 \end{equation}

Let us show that the equations \eqref{banana61} are equivalent to
\begin{equation}\label{banana25}
\begin{array}{c}
{ \begin{array}{c}
    {w=1,\ldots,\gamma}
\end{array}}\\[5pt]
\begin{vmatrix}
  \left(\frac{s_{n+r+1}}{s_{n+r}}\right)(\tilde{\alpha}_w) &  \left(\frac{s_{n+r+1}}{s_{n+r}}\right)'(\tilde{\alpha}_w) &  \cdots & \left(\frac{s_{n+r+1}}{s_{n+r}}\right)^{(k_w-1)}(\tilde{\alpha}_w)  &  V_{\nu-r}\\
  \left(\frac{s_{n+r+2}}{s_{n+r}}\right)(\tilde{\alpha}_w) &  \left(\frac{s_{n+r+2}}{s_{n+r}}\right)'(\tilde{\alpha}_w) &  \cdots &\left(\frac{s_{n+r+2}}{s_{n+r}}\right)^{(k_w-1)}(\tilde{\alpha}_w) &  V_{\nu-r-1}\\
   \vdots   & \vdots    & \vdots &  \vdots &  \vdots\\
  \left(\frac{s_{n+r+k+1}}{s_{n+r}}\right)(\tilde{\alpha}_w) & \left(\frac{s_{n+r+k+1}}{s_{n+r}}\right)'(\tilde{\alpha}_w) &  \cdots &  \left(\frac{s_{n+r+k+1}}{s_{n+r}}\right)^{(k_w-1)}(\tilde{\alpha}_w) & V_{\nu-(r+k)} \\
 \end{vmatrix}
  \end{array}=0,
  \end{equation}
for $r=0,1,\ldots,\nu-2k$ (this expression represents the determinant of order $k+1$ in which the indicated group of columns, evaluated at $\tilde{\alpha}_w$, are successively written out for $w=1,\ldots,\gamma$ and the last column is $[V_{\nu-r} \,\, V_{\nu-r-1} \quad  \cdots \quad  V_{\nu-(r+k)} ]^{T}$). In fact,   expanding \eqref{banana25} along the last column,
we have
$$\sum_{q=1}^{k+1}(-1)^{k+q+1} \tilde{M}_n^r(q,k+1)V_{\nu-(r+q-1)}=0,$$
where  $\tilde{M}_n^r(q,k+1)$ is the $(q,k+1)^{\textup{th}}$ minor of the determinant in \eqref{banana25}. Moreover, it is easy to check that
$$\tilde{M}_n^r(q,k+1)=(-1)^{k-q}\tilde{\Lambda}_{n+r}(q,k+1).$$ Therefore, setting $q=p+1$ in \eqref{banana61}, we obtain
$$0=-\sum_{q=1}^{k+1}\tilde{\Lambda}_{n+r}(q,k+1)V_{\nu-(r+q-1)}=\sum_{q=1}^{k+1}(-1)^{k+q+1} (-1)^{k-q}\tilde{\Lambda}_{n+r}(q,k+1)V_{\nu-(r+q-1)}$$
$$=\sum_{q=1}^{k+1}(-1)^{k+q+1} \tilde{M}_n^r(q,k+1)V_{\nu-(r+q-1)}$$
as we needed to show.

Let us transform \eqref{banana25} further. By the Leibniz rule, we have that for all $j\geq 1$ and $p\geq 0$,
\[ \left(\frac{s_{n+r+j}}{s_{n+\nu+1}}\right)^{(p)} = \sum_{i=0}^p { p \choose i} \left(\frac{s_{n+r}}{s_{n+\nu+1}}\right)^{(p-i)} \left(\frac{s_{n+r+j}}{s_{n+r}}\right)^{(i)}.
\]
Notice that the factors of $\left(\frac{s_{n+r+j}}{s_{n+r}}\right)^{(i)}$ do not depend on $j$. Consequently, taking column operations on the determinant in \eqref{banana25} and having in mind that the determinant equals zero, we obtain that the system \eqref{banana25} is equivalent to
\begin{equation}\label{banana26}
\begin{array}{c}
{ \begin{array}{c}
    {w=1,\ldots,\gamma}
\end{array}}\\[5pt]
\begin{vmatrix}
  \left(\frac{s_{n+r+1}}{s_{n+\nu+1}}\right)(\tilde{\alpha}_w) &  \left(\frac{s_{n+r+1}}{s_{n+\nu+1}}\right)'(\tilde{\alpha}_w) &  \cdots & \left(\frac{s_{n+r+1}}{s_{n+\nu+1}}\right)^{(k_w-1)}(\tilde{\alpha}_w)  &  V_{\nu-r}\\
  \left(\frac{s_{n+r+2}}{s_{n+\nu+1}}\right)(\tilde{\alpha}_w) &  \left(\frac{s_{n+r+2}}{s_{n+\nu+1}}\right)'(\tilde{\alpha}_w) &  \cdots &\left(\frac{s_{n+r+2}}{s_{n+\nu+1}}\right)^{(k_w-1)}(\tilde{\alpha}_w) &  V_{\nu-r-1}\\
   \vdots   & \vdots    & \vdots &  \vdots &  \vdots\\
  \left(\frac{s_{n+r+k+1}}{s_{n+\nu+1}}\right)(\tilde{\alpha}_w) & \left(\frac{s_{n+r+k+1}}{s_{n+\nu+1}}\right)'(\tilde{\alpha}_w) &  \cdots &  \left(\frac{s_{n+r+k+1}}{s_{n+\nu+1}}\right)^{(k_w-1)}(\tilde{\alpha}_w) & V_{\nu-(r+k)} \\
 \end{vmatrix}
  \end{array}=0,
  \end{equation}
  for $r=0,\ldots,\nu-2k$.



We consider \eqref{banana26} as a linear system of $\nu-2k+1$ equations with $\nu-k+1$ unknowns $V_k,\ldots, V_\nu.$
The rank of this system is $\nu-2k+1$ for $n$ sufficiently large. Thus, the null space  has dimension $k.$ Therefore, every solution of \eqref{banana26} can be written as a unique linear combination of $k$ linearly independent solutions $W_1(n),\ldots,W_k(n).$ The structure of \eqref{banana26} easily reveals that for each $w=1,\ldots,\gamma$ and $p= 0,\ldots,k_w-1$
$$W_{j(w,p)}(n) =\left[\left(\frac{s_{n+h+1}}{s_{n+\nu+1}}(\tilde{\alpha}_w)\right)^{(p)}\right]^T_{h=0,\ldots,\nu-k},\quad j(w,p)=\sum_{l=0}^{w-1}k_l+p+1, \quad k_0=0,$$ is a solution of the homogeneous linear system of equations \eqref{banana26}. Moreover, they are linearly independent (for all sufficiently large $n$) because using \eqref{banana8}
\begin{align}\label{banana28}
 &\begin{array}{c}
{ \begin{array}{c}
    {w=1,\ldots,\gamma}
\end{array}}\\[5pt]
\lim_{n\rightarrow \infty} \begin{vmatrix}
 \frac{s_{n+\nu-2k+2}}{s_{n+\nu+1}}(\tilde{\alpha}_w) &\left(\frac{s_{n+\nu-2k+2}}{s_{n+\nu+1}}\right)'(\tilde{\alpha}_w)&  \ldots &\left(\frac{s_{n+\nu-2k+2}}{s_{n+\nu+1}}\right)^{(k_w)}(\tilde{\alpha}_w)\\
  \frac{s_{n+\nu-2k+3}}{s_{n+\nu+1}}(\tilde{\alpha}_w) &\left(\frac{s_{n+\nu-2k+3}}{s_{n+\nu+1}}\right)'(\tilde{\alpha}_w)&  \ldots &\left(\frac{s_{n+\nu-2k+3}}{s_{n+\nu+1}}\right)^{(k_w)}(\tilde{\alpha}_w)\\
\vdots  & \vdots  & \cdots & \vdots\\
 \frac{s_{n+\nu-k+1}}{s_{n+\nu+1}}(\tilde{\alpha}_w)   & \left(\frac{s_{n+\nu-k+1}}{s_{n+\nu+1}}\right)'(\tilde{\alpha}_w)    & \ldots &  \left(\frac{s_{n+\nu-k+1}}{s_{n+\nu+1}}\right)^{(k_w)}(\tilde{\alpha}_w) \\
 \end{vmatrix}
  \end{array}\notag\\
 & = (-1)^{(k-1)k/2}\prod_{w=1}^{\gamma}  (k_w-1)!! (\Phi^{k_wk})(\tilde{\alpha_w})(\Phi'(\tilde{\alpha_w}))^{\frac{(k_w-1)k_w}{2}}  \prod_{1\leq i< j \leq \gamma}\left(\Phi(\tilde{\alpha_j})-\Phi(\tilde{\alpha}_i)\right)^{k_ik_j} \notag\\
 &\not=0.
  \end{align}

Since $$V_{\nu-(r+p)}:=(-1)^{r+p} \frac{D(n+p+r,\nu-r-p)}{\prod_{j=r+p}^{\nu-1}a_{1,n+j}},$$
there exists a unique set of coefficients $C_1(n),\ldots,C_k(n$) such that
\[
\left[(-1)^{h} \frac{D(n+h,\nu-h)}{\prod_{j=h}^{\nu-1}a_{1,n+j}} \right]_{h=0,\ldots, \nu-k}=\sum_{j=1}^k C_j(n) W_j(n).
\]
Thus,
\begin{equation}\label{banana27}
 \sum_{w=1}^{\gamma}\sum_{p=0}^{k_w-1} {c}_{n,\nu}(w,p)\left(\frac{s_{n+h+1}}{s_{n+\nu+1}}\right)^{(p)}(\tilde{\alpha}_w)=(-1)^{h} \frac{D(n+h,\nu-h)}{\prod_{\tau=h}^{\nu-1}a_{1,n+\tau}}, \quad \quad h=0,\ldots, \nu-k,
\end{equation}
where the constants ${c}_{n,\nu}(w,p)$ are uniquely determined.

To estimate the ${c}_{n,\nu}(w,p), w=1,\ldots,\gamma, p=0,\ldots,k_w-1$, we use the linear system of equations
\begin{equation}\label{banana66}
 \sum_{w=1}^{\gamma}\sum_{p=0}^{k_w-1} {c}_{n,\nu}(w,p)\left(\frac{s_{n+h+1}}{s_{n+\nu+1}}\right)^{(p)}(\tilde{\alpha}_w)=(-1)^{h} \frac{D(n+h,\nu-h)}{\prod_{\tau=h}^{\nu-1}a_{1,n+\tau}},
\end{equation}
corresponding to $ h=\nu-2k+1,\ldots, \nu-k$. From \eqref{banana28}, it follows that the determinant of this system is different from zero for all sufficiently large $n$. From \eqref{banana19} and \eqref{banana20}, it is not difficult to verify that
\begin{equation}\label{banana63}
\left|\frac{D(n+h,\nu-h)}{\prod_{\tau=h}^{\nu-1} a_{1,n+\tau}} \right| \leq c_6, \quad h=\nu-2k+1,\ldots,\nu-k,
\end{equation}
From \eqref{banana8} and the Weierstrass theorem we have that
\begin{equation} \label{c}
\lim_{n \to \infty} \left(\frac{s_{n+h+1}(z)}{s_{n+\nu+1}(z)}\right)^{(p)} = \left(\Phi^{\nu-h}(z)\right)^{(p)}
\end{equation}
uniformly inside $\overline{\mathbb{C}} \setminus E$. Therefore, the coefficients of system \eqref{banana66} remain uniformly bounded with respect to $n$ or $\nu$ since in those equations $k \leq \nu-h \leq 2k -1$. Applying Cramer's rule and \eqref{banana63}, it follows that
\begin{equation} \label{a}
|{c}_{n,\nu}(w,p)|\leq c_7, \quad \quad w=1,\ldots,\gamma, \quad p=0,\ldots,k_w-1,
\end{equation}
where $c_7$ does not depend on $n$ or $\nu$.
Taking $h=0$ in \eqref{banana27}, we have
\begin{equation} \label{b}
 \frac{D(n,\nu)}{\prod_{\tau=0}^{\nu-1}a_{1,n+\tau}}= \sum_{w=1}^{\gamma}\sum_{p=0}^{k_w-1} {c}_{n,\nu}(w,p)\left(\frac{s_{n+1}}{s_{n+\nu+1}}\right)^{(p)}(\tilde{\alpha}_w)
\end{equation}
From \eqref{a}, \eqref{b}, and \eqref{c} with $h=0$ it follows for any $\varepsilon > 0$ there exists $n_0$ such that for $n \geq n_0$
\begin{equation} \label{d}
\left|\frac{D(n,\nu)}{\prod_{\tau=0}^{\nu-1} a_{1,n+\tau}} \right| \leq c_{8}  ({\sigma_\gamma+\varepsilon})^{\nu}, \qquad \nu =0,\ldots,\nu_0 -1.\end{equation}
(Notice that using Cauchy's integral formula it is easy to prove that $|\left(\Phi^{\nu}(\widetilde{\alpha}_w)\right)^{(p)}| \leq c_9| \Phi^{\nu}(\sigma_\gamma + \varepsilon) |$). Now, \eqref{banana19}, \eqref{banana30}, and \eqref{d} give
\begin{align}\label{banana31}
|[Q_mF]_{n+1}|& \leq \sum_{\nu=0}^{\nu_0-1} |B_{n+\nu}(\rho)| \left|\frac{D(n,\nu)}{\prod_{\tau=0}^{\nu-1}a_{1,n+\tau}} \right|\frac{1}{|a_{1,n+\nu}|}\notag\\
&\leq c_{10} \sum_{\nu=0}^{\nu_0-1} |B_{n+\nu}(\rho)| (\sigma_\gamma+\varepsilon)^{\nu}.
\end{align}

Next, let us bound $|B_{n+\nu}(\rho)|.$ Take $\varepsilon >0$ such that $\sigma_\gamma + \varepsilon < \rho - \varepsilon$ and $\delta < \delta' <1$. From \eqref{banana51}, \eqref{banana10}, and \eqref{e}, we have for all sufficiently large $n$
\[
|\eta(\rho,n,j)|\leq c_{11} (\delta')^n \frac{1}{(\rho-\varepsilon)^{n+j}},
\]
and
\[
|[Q_mF]_{n+\nu+j}| \leq \frac{c_{12}}{(\rho-\varepsilon)^{n+\nu+j}}.
\]
Thus, from \eqref{banana19}, \eqref{banana20} and the definition of $B_{n+\nu}(\rho)$ we obtain
\begin{align}\label{banana34}
&|B_{n+\nu}(\rho)| \leq c_{13} \left(\frac{\delta'}{\rho-\varepsilon}\right)^{n+\nu}  \sum_{j=1}^{k+1} \frac{1}{(\rho-\varepsilon)^{j}}\notag\\
&=c_{14} \left(\frac{\delta'}{\rho-\varepsilon}\right)^{n}\left(\frac{1}{\rho-\varepsilon}\right)^{\nu}  ,\quad \quad \nu=0,1,\ldots,\nu_0-k-1,
\end{align}
and
\begin{align}\label{banana35}
&|B_{n+\nu}(\rho)|\leq c_{14} \left(\frac{\delta'}{\rho-\varepsilon}\right)^{n}\left(\frac{1}{\rho-\varepsilon}\right)^{\nu} + \sum_{j=\nu_0-\nu+1}^{k+1} \frac{c_{15}}{(\rho-\varepsilon)^{n+\nu+j}}\notag\\
&\leq c_{14} \left(\frac{\delta'}{\rho-\varepsilon}\right)^{n}\left(\frac{1}{\rho-\varepsilon}\right)^{\nu}+c_{16}\frac{1}{(\rho-\varepsilon)^{n+\nu}}\notag\\
&\leq \frac{c_{17}}{(\rho-\varepsilon)^{n+\nu}}, \quad \quad \quad\quad \quad \quad \quad \quad \quad\quad  \nu=\nu_0-k,\ldots,\nu_0-1.
\end{align}
Applying \eqref{banana34} and \eqref{banana35} to \eqref{banana31}, we have
$$|[Q_mF]_{n+1}|\leq c_{18}  \left( \left(\frac{\delta'}{\rho-\varepsilon} \right)^n \sum_{\nu=0}^{\nu_0-k-1} \left(\frac{\sigma_\gamma+\varepsilon}{\rho-\varepsilon} \right)^{\nu}+\frac{1}{(\rho-\varepsilon)^n}\sum_{\nu=\nu_0-k}^{\nu_0-1} \left( \frac{\sigma_\gamma+\varepsilon}{\rho-\varepsilon}\right)^{\nu}\right)$$
Setting $\theta=(\sigma_\gamma+\varepsilon)/(\rho-\varepsilon)<1,$ we find that
$$|[Q_mF]_{n+1}|\leq c_{19}  \left(  \left(\frac{\delta'}{\rho-\varepsilon} \right)^{n+1} \sum_{\nu=0}^{\infty} \theta^{\nu} +\frac{1}{(\rho-\varepsilon)^n}\sum_{\nu=\nu_0-k}^{\infty} \theta^{\nu} \right) $$
Letting $\nu_0 \rightarrow \infty,$ we obtain
$$|[Q_mF]_{n+1}|\leq c_{20}  \left(\frac{\delta'}{\rho-\varepsilon} \right)^{n+1},$$
and
$$\varlimsup_{n \rightarrow \infty}|[Q_mF]_{n+1}|^{1/(n+1)}\leq  \frac{\delta'}{\rho-\varepsilon}.$$
Making $\varepsilon \rightarrow 0,$ $\delta'\rightarrow \delta,$ and $\rho \rightarrow \rho_m(F),$ we obtain the claim that
$$\varlimsup_{n \rightarrow \infty}|[Q_mF]_{n+1}|^{1/(n+1)}\leq  \frac{\delta}{\rho_m(F)}.$$ From this and \eqref{banana9}, if follows that the function $Q_{m}F$ is holomorphic in $D_{\rho_m(F)/\delta}$. Thus $F$ is meromorphic with at most $m$ poles on $D_{\rho_m(F)/\delta}$ which contradicts the definition of $\rho_m(F)$ unless $\rho_m(F) = \infty$.

In the final step, we show that if $F$ is meromorphic in $\mathbb{C}$ and has $k < m$ poles, then $F$ is a rational function.  In fact, in that case
$$F:=F^{*}+R_k,$$ where $F^{*}$ is an entire function and $R_k$ is a rational function with $k$ poles at $\lambda_1,\ldots,\lambda_k.$
Applying the residue theorem and  arguing as in \eqref{banana40}, we obtain
 $$[R_k Q_{n,m}^{\mu}]_{n+b}=\frac{1}{2\pi i} \int_{\Gamma_\rho} R_k(t) Q_{n,m}^{\mu}(t) s_{n+b}(t) dt -\sum_{w=1}^{\gamma} \textup{res}(R_k Q_{n,m}^{\mu}s_{n+b},\tilde{\alpha}_w)$$
 \begin{equation}\label{banana70}
 =\frac{1}{2\pi i} \int_{\Gamma_\rho} R_k(t) Q_{n,m}^{\mu}(t) s_{n+b}(t) dt -\sum_{w=1}^{\gamma}\sum_{p=0}^{k_w-1}\xi_n(w,p) \left(\frac{s_{n+b}}{s_{n+1}} \right)^{(p)}(\tilde{\alpha_w}),
 \end{equation}
where $\rho>\sigma_\gamma$  and
$$
\xi_n(w,p)=\frac{1}{(k_w-1)!}{k_w-1 \choose p} \lim_{z \rightarrow \tilde{\alpha}_w} ((z-\tilde{\alpha}_w)^{k_w}R_k(z) Q_{n,m}^{\mu}(z)  s_{n+1}(z))^{(k_w-1-p)}.
$$
Since $s_{n+b}$ has a zero of order $n+b+1$ at infinity and $\deg{(Q_{n,m}^{\mu})}\leq m,$ for $n$ sufficiently large, we have
\begin{equation}\label{banana71}
\frac{1}{2\pi i} \int_{\Gamma_\rho} R_k(t) Q_{n,m}^{\mu}(t) s_{n+b}(t) dt=0.
\end{equation}

By the definition of Pad\'e-orthogonal approximants,
$$0=[FQ_{n,m}^{\mu}]_{n+b}=[F^{*}Q_{n,m}^{\mu}]_{n+b}+[R_kQ_{n,m}^{\mu}]_{n+b},\quad \quad b=1,\ldots,m.$$
Since $k+1 \leq m$, using \eqref{banana70} and \eqref{banana71}, we have
$$[F^{*}Q_{n,m}^{\mu}]_{n+b}=\sum_{w=1}^{\gamma}\sum_{p=0}^{k_w-1}\xi_n(w,p) \left(\frac{s_{n+b}}{s_{n+1}} \right)^{(p)}(\tilde{\alpha_w}),\quad \quad b=1,\ldots,k+1.$$

Arguing as above in the deduction of \eqref{banana12}-\eqref{banana150}, we obtain
$$[F^*Q_{n,m}^{\mu}]_{n+1}=\sum_{j=2}^{k+1} a_{j,n} [F^{*}Q_{n,m}^{\mu}]_{n+j},$$
 where $a_{j,n}:=\Lambda_{n+1}(j-1,1)/\Lambda_{n+1}$, $\Lambda_{n+1}$ is matrix \eqref{matrixL} with $n$ replaced by $n+1$,
 and $\Lambda_{n+1}(j-1,1)$ is the determinant obtained from $\Lambda_{n+1}$ replacing row $j-1$  by the row
 $$
\begin{array}{c}
{ \begin{array}{c}
    {w=1,\ldots,\gamma}
\end{array}}\\[5pt]
\left[1\quad  0\quad 0\quad 0\quad  \quad\ldots \quad\quad 0\right].
  \end{array}
$$
It is easy to verify that $|a_{j,n}| \leq c_{21},$ for all $j=2,\ldots,k+1.$ Therefore,
\begin{equation}\label{banana72}
|[F^* Q_{n,m}^{\mu}]_{n+1}|\leq c_{21} \sum_{j=2}^{k+1}  |[F^* Q_{n,m}^{\mu}]_{n+j}|.
\end{equation}

Let $$Q_{n,m}^{\mu}(z):=z^m+\sum_{j=0}^{m-1} q_{n,j} z^j$$
and $$F^{*}(z):=\sum_{\nu=0}^{\infty} F^{*}_\nu p_{\nu}(z),$$
where $F^{*}_\nu:=\langle F^{*}, p_\nu \rangle_{\mu}.$
Note that the series $\sum_{\nu=0}^{\infty} F_\nu^{*} p_\nu$ converges to $F^{*}$ uniformly inside $\mathbb{C}$ and $\lim_{\nu\rightarrow \infty} |F^{*}_{\nu}|^{1/\nu}=0$ because $F^{*}$ is an entire function. Therefore, for all $b=1,\ldots,k+1,$
$$[ Q_{n,m}^{\mu} F^{*}]_{n+b}=\langle  Q_{n,m}^{\mu} F^{*}, p_{n+b} \rangle_{\mu}=\langle  z^{m} F^{*}, p_{n+b}\rangle_{\mu}+\sum_{j=0}^{m-1} q_{n,j} \langle  z^{j} F^{*}, p_{n+b}\rangle_{\mu}$$
$$=\sum_{\nu=0}^{\infty} F^{*}_\nu  \langle z^m p_{\nu} ,p_{n+b} \rangle_{\mu}+\sum_{j=0}^{m-1} q_{n,j} \sum_{\nu=0}^{\infty} F^{*}_{\nu} \langle z^{j}p_{\nu} , p_{n+b}\rangle_{\mu}$$
$$=\sum_{\nu=n+b-m}^{\infty} F^{*}_\nu  \langle z^m p_{\nu} ,p_{n+b} \rangle_{\mu}+\sum_{j=0}^{m-1} q_{n,j} \sum_{\nu=n+b-j}^{\infty} F^{*}_{\nu} \langle z^{j}p_{\nu} , p_{n+b}\rangle_{\mu}.$$
By the Cauchy-Schwarz inequality and the orthonormality of $p_\nu$, for all $n,\nu,b\in \mathbb{N}$ and $j=1,\ldots,m,$
$$|\langle z^{j}p_{\nu} , p_{n+b}\rangle_{\mu}| \leq c_{22}.$$ Using \eqref{banana51}, it follows that $|q_{n,j}| \leq c_{23}$ and therefore
\begin{equation}\label{banana73}
|[Q_{n,m}^{\mu} F^{*} ]_{n+b}| \leq c_{24} \sum_{\nu=n+b-m}^{\infty} |F^{*}_\nu| \leq c_{24} \sum_{\nu=n+2-m}^{\infty} |F^{*}_\nu|, \quad \quad b=2,\ldots k+1.
\end{equation}
Moreover,
$$[ Q_{n,m}^{\mu} F^{*}]_{n+1}=\sum_{\nu=n+1-m}^{\infty} F^{*}_\nu  \langle z^m p_{\nu} ,p_{n+1} \rangle_{\mu}+\sum_{j=0}^{m-1} q_{n,j} \sum_{\nu=n+1-j}^{\infty} F^{*}_{\nu} \langle z^{j}p_{\nu} , p_{n+1}\rangle_{\mu}.$$
\begin{equation}\label{banana74}
=\frac{\kappa_{n+1-m}}{\kappa_{n+1}}F^{*}_{n+1-m}+\sum_{\nu=n+2-m}^{\infty} F^{*}_\nu  \langle z^m p_{\nu} ,p_{n+1} \rangle_{\mu}+\sum_{j=0}^{m-1} q_{n,j} \sum_{\nu=n+1-j}^{\infty} F^{*}_{\nu} \langle z^{j}p_{\nu} , p_{n+1}\rangle_{\mu}.
\end{equation}
Combining \eqref{banana72}, \eqref{banana73}, and \eqref{banana74}, we have
$$\frac{\kappa_{n+1-m}}{\kappa_{n+1}}|F^{*}_{n+1-m}| \leq c_{25}\sum_{\nu=n+2-m}^{\infty} |F^{*}_\nu|.$$ By Lemma \ref{kappa},
$$\lim_{n \rightarrow \infty} \frac{\kappa_{n+1-m}}{\kappa_{n+1}}={\textup{cap}(E)^{m}}>0;$$ therefore, there exists $n_1\geq 1$ such that for all $n \geq n_1,$
$$\frac{\kappa_{n+1-m}}{\kappa_{n+1}}\geq c_{26}>0.$$ Setting $N=n+1-m,$ we obtain
$$|F_N^{*}|\leq c_{27} \sum_{\nu=N+1}^{\infty} |F_N^{*}|, \quad N \geq N_0.$$
By Lemma \ref{trick}, there exist $N_1\in \mathbb{N}$ such that $F_N^{*}=0$ for all $N \geq N_1.$ Therefore, $F^*$ is a polynomial and $F$ is a rational function with at most $k$ poles. However, in this case it is easy to see from \eqref{averbvkaer} that under appropriate column operations $\Delta_{n,m}(F, \mu)=0$ for all $n$ sufficiently large. This contradicts the  assumption that for all $n$ sufficiently large, $\deg(Q_{n,m}^{\mu})=m.$ Consequently, $F$ has $m$ poles in $D_{\rho_m(F)}.$

By Theorem A, we conclude that $\lambda_1,\ldots, \lambda_m$ are poles of $F$ in $D_{\rho_m(F)}.$ To prove \eqref{banana75}, let us consider the region $D_{\rho_{m-1}(F)}.$ Notice that $\rho_{m-1}(F):=\max_{j=1,\ldots,m} |\Phi(\lambda_j)|.$ Clearly, $F$ has less than $m$ poles in $D_{\rho_{m-1}(F)}.$ Repeating the proof above we obtain that $\varlimsup_{n \rightarrow \infty} [Q_{m}F]_n^{1/n}\leq \delta/\rho_{m-1}(F)$. This implies that $F$ is meromorphic with at most $m$ poles in $D_{\rho_{m-1}(F)/\delta}.$ From the definition of $\rho_m(F)$ this implies that  $$\rho_m(F) \geq \frac{1}{\delta}{\max_{j=1\ldots,m}|\Phi(\lambda_j)|}.$$ This completes the proof.
\end{proof}

\noindent Nattapong Bosuwan\\
         Department of Mathematics\\
         Faculty of Science\\
         Mahidol University\\
         Rama VI Road, Ratchathewi District,\\
         Bangkok 10400, Thailand\\
         \noindent email: nattapong.bos@mahidol.ac.th,\\

\noindent G. L\'opez Lagomasino\\
Departamento de Matem\'aticas \\
Universidad Carlos III de Madrid \\
c/ Avda. de la Universidad, 30\\
28911, Legan\'es, Spain\\
\noindent email: lago@math.uc3m.es
}
\end{document}